\documentclass[preprint,12pt]{elsarticle}

\usepackage{graphicx}
\usepackage{amssymb}
\usepackage[latin1]{inputenc}
\usepackage{amsmath}
\usepackage{amsfonts}
\usepackage{graphicx}
\usepackage{color}
\usepackage{verbatim}
\usepackage{amsthm}

\newcommand{\be}{\begin{eqnarray}}
\newcommand{\e}{\end{eqnarray}}
\newcommand{\bex}{\begin{eqnarray*}}
\newcommand{\ex}{\end{eqnarray*}}
\newcommand{\mbb}{\mathbb}

\newcommand{\mbf}{\mathbf}
\newcommand{\mcal}{\mathcal}
\newcommand{\la}{\langle}
\newcommand{\ra}{\rangle}
\newcommand{\p}{\partial}

\newtheorem{theo}{Theorem}%[section]

\newtheorem{lem}[theo]{Lemma}
\newtheorem{prop}[theo]{Proposition}

\newtheorem{cor}[theo]{Corollary}

\newtheorem*{theo*}{Theorem}
\begin{document}

\begin{frontmatter}

%% Title, authors and addresses

%% use the tnoteref command within \title for footnotes;
%% use the tnotetext command for the associated footnote;
%% use the fnref command within \author or \address for footnotes;
%% use the fntext command for the associated footnote;
%% use the corref command within \author for corresponding author footnotes;
%% use the cortext command for the associated footnote;
%% use the ead command for the email address,
%% and the form \ead[url] for the home page:
%%
\title{The total mass of super-Brownian motion upon exiting balls and Sheu's compact support condition}
%% \title{Title\tnoteref{label1}}
%% \tnotetext[label1]{}
 \author{Marion Hesse}
 \ead{m.hesse@bath.ac.uk}
 \author{Andreas Kyprianou}
  \ead{a.kyprianou@bath.ac.uk}
%% \ead[url]{home page}
%% \fntext[label2]{}
%% \cortext[cor1]{}
 \address{University of Bath}
%% \fntext[label3]{}

%% use optional labels to link authors explicitly to addresses:
%% \author[label1,label2]{<author name>}
%% \address[label1]{<address>}
%% \address[label2]{<address>}

\begin{abstract}

We study the total mass of a $d$-dimensional super-Brownian motion as it first exits an increasing sequence of balls.
The process of the total mass  is a time-inhomogeneous continuous-state branching process, where the increasing radii of the  balls are taken as the time parameter. We are able to characterise its time-dependent branching mechanism and show that it converges, as time goes to infinity, towards the branching mechanism of the total mass of a one-dimensional super-Brownian motion as it first crosses above an increasing sequence of levels. \\
Our results allow us to identify the compact support criterion given in Sheu (1994) as a classical Grey condition (1974) for the aforementioned limiting branching mechanism. 
\end{abstract}

\begin{keyword}
Super-Brownian motion, exit measures \sep time-dependent continuous state branching processes \sep compact support condition.
%% keywords here, in the form: keyword \sep keyword

\MSC[2000] 60J68 %superprocesses
\sep 60J80. %branching processes
%% MSC codes here, in the form: \MSC code \sep code
%% or \MSC[2008] code \sep code (2000 is the default)

\end{keyword}

\end{frontmatter}

% \linenumbers
%%%%%%

\section{Introduction and main results}
Suppose that $X=(X_{t},{t}\geq 0)$ is a super-Brownian motion in $\mbb{R}^d$, $d\geq 1$, with 
general branching mechanism $\psi$ of the form 
\be\label{eq-branchingmechpsi}
\psi(\lambda) = - \alpha \lambda + \beta \lambda^2 + \int_{(0,\infty)} (e^{-\lambda x} - 1 + \lambda x ) \Pi({\rm d}x), \ \ \lambda\geq 0,
\e
where $\alpha=-\psi'(0+) \in (-\infty,\infty)$, $\beta\geq 0$ and $\Pi$ is a measure concentrated on $(0,\infty)$ which satisfies $\int_{(0,\infty)} (x\wedge x^2) \Pi({\rm d}x) <\infty$. Assume $\psi(\infty)=\infty$.
Denote by $\mbf{P}_\mu$ the law of $X$ with initial configuration according to $\mu\in\mcal{M}_F(
\mbb{R}^d)$, the space of finite measures on $\mbb{R}^d$ with compact support. We write $\mcal{M}_F(D)$ for the space of finite measures supported on $D\subset \mbb{R}^d$.\\
A construction of superprocesses with a general branching mechanism $\psi$ as in (\ref{eq-branchingmechpsi}) can be found in  Fitzsimmons \cite{fitzsimmons1988construction}, see also Section 2.3 in Li \cite{libook} which provides a comprehensive account on the theory of superprocesses.\\ 
We call $X$ (sub)critical if $\psi'(0+)\geq0$ and supercritical if $\psi'(0+)< 0$.  
Denote the root of $\psi$ by $\lambda^*:=\inf\{\lambda\geq 0 : \psi(\lambda)>0 \}$. In the (sub)critical case, we have $\lambda^*=0$. In the supercritical case, convexity of $\psi$ and the condition $\psi(\infty)=\infty$ ensure that there is a unique and finite $\lambda^*>0$. In both cases, 
\bex
\mbf{P}_{\mu}( \lim_{t\to\infty} ||X_t||=0 ) = e^{-\lambda^* ||\mu||},
\ex
where $||\mu||$ denotes the total mass of the measure $\mu\in\mcal{M}_F(\mbb{R}^d)$. 

We want to study the total mass of the super-Brownian motion $X$ upon its first exit from an increasing sequence of balls. Fix an initial radius $r>0$ and let $D_{{s}}:=\{x\in \mbb{R}^d : ||x||< {s}\}$ be the open  ball of radius ${s}\geq{r}$ around the origin. According to Dynkin's theory of exit measures \cite{dynkinbemsp01}, we can describe the mass of $X$ as it first exits the growing sequence of balls $(D_{{s}}, {s}\geq{r})$ as a sequence of random measures on $\mbb{R}^d$, known as branching Markov exit measures. We denote this sequence of branching Markov exit measures by $\{X_{D_{s}}, {s}\geq r\}$. Informally, the  measure $X_{D_{s}}$  is  supported on the boundary $\p D_{s}$ and it is obtained by `freezing' mass of the super-Brownian motion when it first hits 
$\p D_{{s}}$.  \\
Formally,  $\{X_{D_{s}}, {s}\geq r\}$ is characterised by the following branching Markov property, see for instance Section 1.1 in Dynkin and Kuznetsov \cite{dynkinkuznetsovnmeasures04}.  Let $\mu\in\mcal{M}_F(D_r)$ and, for ${z}\geq {r}$, define $\mcal{F}_{D_z} :=\sigma(X_{D_{z'}}, r\leq {z'} \leq {z})$.  For any positive, bounded, continuous function $f$ on $\p D_{s}$,
\be\label{eq-strongmarkov}
\mbf{E}_{\mu}[e^{-\la f, X_{D_{s}}\ra}|\mcal{F}_{D_{z}}] = e^{- \la v_f(\cdot,{s}), X_{D_{z}} \ra}, \ \ \ 0<r \leq {z}\leq {s},
\e
where the Laplace functional $v_f$ is the unique non-negative solution to
\be\label{eq-semigroupexitmeasures}
v_f(x,{s})= \mbb{E}_{x}[f(\xi_{T_{s}})] - \mbb{E}_{x}\Big[ \int_0^{T_{s}}  \psi(v_f(\xi_{z},{s})) \ {\rm d}{z}\Big],
\e
and $((\xi_{z},{z}\geq 0),\mbb{P}_{x})$ is an $\mbb{R}^d$-Brownian motion with $\xi_0=x$ and with $T_{s}:=\inf\{{z}>0: \xi_{z} \notin D_{s} \}$ denoting its first exit time  from $D_{s}$.
In (\ref{eq-strongmarkov}), we have used the inner product notation $\la f, \mu \ra =\int_{\mbb{R}^d} f(x) \mu({\rm d}x) $. 

For ${s}\geq r$, let $Z_{s}:=||X_{D_{s}}||$  denote the total mass that is `frozen' when it first hits  the boundary of the ball $D_s$. We can then define the total mass process 
$(Z_{s},{s}\geq r)$ which uses the radius $s$ as its time-parameter. Let us write $P_{r}$, for the law of the process $(Z_{s}, {s}\geq r)$ starting at time $r>0$ with unit initial mass.  In case we start with non-unit initial mass $a>0$ we shall use the notation $P_{{a,r}}$ for its law. \\
It is not difficult to see that $Z$ is  a time-inhomogeneous continuous-state branching process and we can characterise it as follows.

\begin{theo}\label{theo-branchingmechpde} 
(i) Let $r>0$.  The process $Z=(Z_{s},{s}\geq r)$ is a time-inhomogeneous continuous-state branching process. This is to say it is a $[0,\infty]$-valued strong Markov process with c\`adl\`ag paths satisfying the branching property
\bex
E_{(a+a'),{r}} [e^{-\theta Z_{s}}] 
&=& E_{a ,r} [e^{-\theta Z_{s}}]E_{a' ,{r}} [e^{-\theta Z_{s}}], 
\ex 
for all $a,a' >0$, $\theta\geq 0$ and ${s}\geq r$.\\ 
(ii) Let $r>0$ and $\mu\in\mcal{M}_F(\partial D_r)$ with $||\mu||=a$. Then, for ${s}\geq r$, we have
\be\label{eq-laplacedef}
E_{{a,r}}[e^{-\theta Z_{s}}]
&=&e^{- u({r},{s},\theta) a} , \ \ \theta\geq 0, 
\e
where the Laplace functional $u({r},{s},\theta)$  satisfies 
\be\label{eq-inhomlaplacefunctional}
u({r},{s},\theta) = \theta - \int_r^{s} \Psi({z},u({z},{s},\theta)) \ {\rm d}{z}, 
\e
for a family of branching mechanisms $(\Psi(r,\cdot), r>0)$ of the form
\be\label{eq-Psi}
\Psi(r,\theta) &=& - q_r + a_r \theta + b_r \theta^2 + \int_{(0,\infty)} (e^{-\theta x} - 1 + \theta x \mbf{1}_{(x<1)}) \Lambda_r({\rm d}x),  
\e
for $\theta\geq 0$, and for each $r>0$ we have $q_r\geq 0$, $a_r\in \mbb{R}$, $b_r\geq 0$ and $\Lambda_r$ is a measure concentrated on $(0,\infty)$ satisfying $\int_{(0,\infty)} (1\wedge x^2) \Lambda_r({\rm d}x)<\infty$.\\
(iii) The branching mechanism $\Psi$ satisfies the PDE
\be\label{eq-pderelationforpsi}
\frac{\partial}{\partial r} \Psi(r,\theta) + \frac{1}{2}\frac{\partial}{\partial \theta} \Psi^2(r,\theta) + \frac{d-1}{r} \Psi(r,\theta) &=& 2 \psi(\theta) \ \ \ \ r>0, \ \theta\in (0,\infty) 
  \nonumber\\
 \Psi(r,\lambda^*)&=& 0, \ \ \ \  r>0.
\e
\end{theo}
The authors are not aware of a result in the literature which states that the definition of the time-dependent CSBP in (i) implies the characterisation in (ii). 
It is therefore outlined in the proof of Theorem \ref{theo-branchingmechpde} (ii) in Section \ref{sec-proofoftheo1} how this implication can be derived as a generalisation of the equivalent result for standard CSBPs in Silverstein \cite{silverstein68}. \\
As part of Theorem \ref{theo-branchingmechpde}, we later prove  that the root $\lambda^*$ of $\psi$ is also the root for each  $\Psi(r,\cdot)$, $ r>0$, cf.  Lemma \ref{lem-rootlambdastar}. This will be a key property for the forthcoming analysis of the family of branching mechanism  $(\Psi(r,\cdot),r>0)$.

Let us  now describe how $\Psi$ changes as $r$ increases. We observe the following change in the shape of the branching mechanism, see Figure \ref{fig-changingshape}.
\begin{prop}\label{prop-changeshape} 
(i) For (sub)critical $\psi$, we have, for $0<r\leq s$, 
\bex
\Psi(r,\theta)\leq \Psi(s,\theta) \ \  \text{ for all } \ \theta\geq 0.
\ex
(ii) For supercritical $\psi$, we have, for $0<r\leq s $, 
\bex
&&  \Psi(r,\theta)\geq \Psi(s,\theta) \ \ \text{ for all } \  \theta\leq \lambda^* \\ 
&& \Psi(r,\theta)\leq \Psi(s,\theta) \ \ \text{ for all } \ \theta\geq \lambda^*. 
\ex

\begin{figure}
\centering
\includegraphics[trim=6cm 12cm 3cm 9cm, clip=true, scale=1]{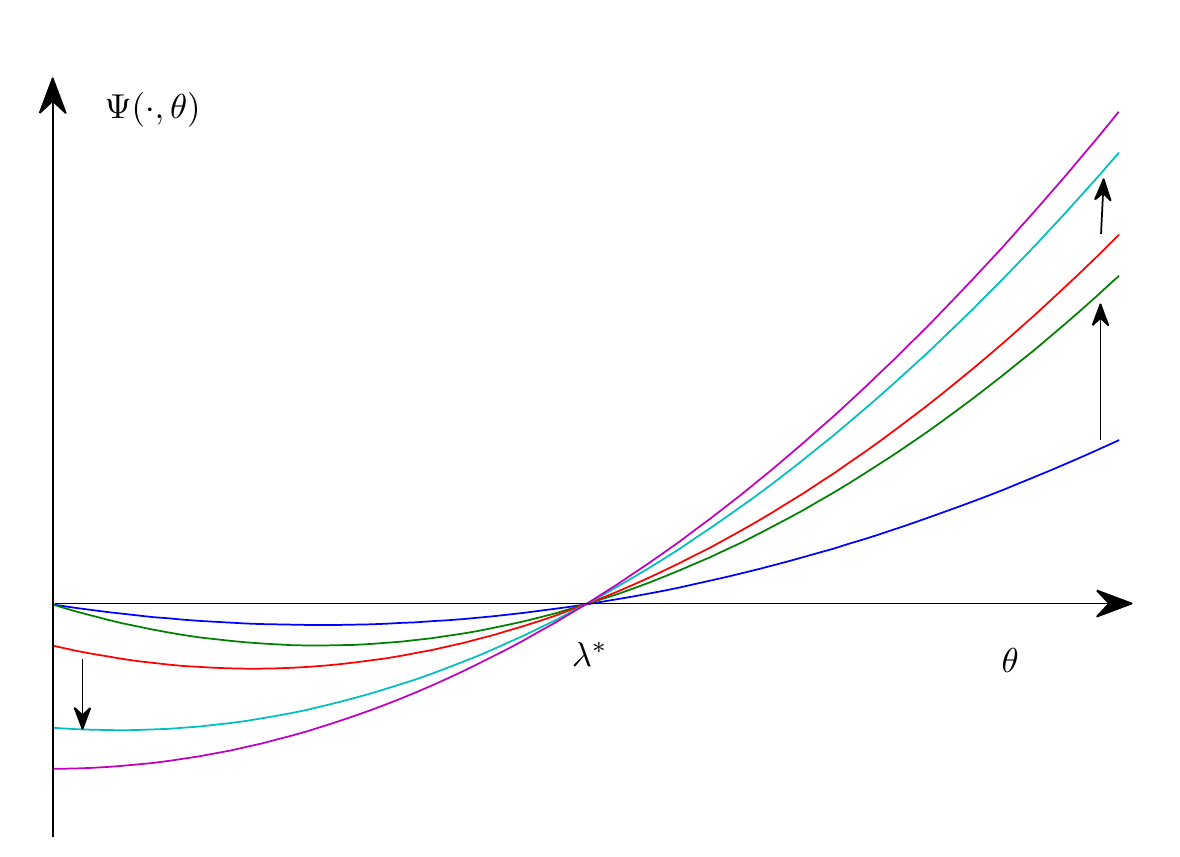}
\caption{Shape of the branching mechanism $\Psi(r,\cdot)$ as $r\to\infty$ in the supercritical case}
 \label{fig-changingshape}
\end{figure}

\end{prop}
This result suggests that there is a limiting branching mechanism ${\Psi}_\infty(\cdot) := \lim_{r\to\infty} \Psi(r,\cdot)$.  Intuitively speaking, in the case where the initial mass is supported on a large ball, the local behaviour of the super-Brownian motion when exiting increasingly larger balls should look like a one-dimensional super-Brownian upon crossing levels. This idea is  supported by the following result. 

\begin{theo}\label{theo-branchingmechlimit} 
For each $\theta\geq 0$, the limit $\lim_{r\uparrow \infty}\Psi(r,\theta)={\Psi}_\infty(\theta)$  is finite and the convergence holds uniformly in $\theta$ on any bounded, closed subset of $\mbb{R}_+$. 
\\
(i) For any $\theta\geq 0$, we have
\be\label{eq-psibartheta}
{\Psi}_\infty(\theta) = 2 \operatorname{sgn}(\psi(\theta)) \  \sqrt{\int_{\lambda^*}^\theta \psi(\lambda) \ {\rm d}\lambda }, 
\e
with $\lambda^*=0$ in the (sub)critical case.\\
(ii) Denote by  $(({Z}^\infty_{s}, {s}\geq 0),{P}^\infty)$ the standard CSBP associated with the limiting branching mechanism ${\Psi}_\infty$, with unit initial mass at time $0$. \\
Then, $({Z}^\infty_{s},{s}\geq 0 )$ is the total mass of the process of branching Markov exit measures of a one-dimensional super-Brownian  motion with unit initial mass at time zero as it first exits the family of intervals $((-\infty, {s}), {s}\geq 0)$. \\
Further,  for any $s>0$,  $\theta\geq 0$,
\be\label{eq-convergencetozinfty}
\lim_{r\to\infty} E_r[e^{-\theta Z_{r+s}}] = {E}^\infty[e^{-\theta {Z}^\infty_{s}}]. 
\e
\end{theo}

Let us remark that, in the supercritical case, the limiting branching mechanism ${\Psi}_\infty$ is critical and possesses an explosion coefficient, that is ${\Psi}_{\infty}'(0+) = 0$ and ${\Psi}_{\infty}(0)<0$.  Thanks to the uniform continuity in $\theta$, this implies  that $\Psi(t,0)<0$ for all sufficiently large $t$. \\
The limiting process $Z^\infty$ in Theorem \ref{theo-branchingmechlimit} has already been studied in Theorem 3.1 in Kyprianou et al. \cite{kyprianouetalsbmtw12}. Note that therein the underlying Brownian motion has a positive drift  which is chosen such that the resulting branching mechanism is conservative. The characterisation can easily be adapted to the driftless case as in Theorem \ref{theo-branchingmechlimit} (ii).  Kaj and Salminen \cite{kaj1993note, kaj1993first} studied the analogous process in the setting of branching particle diffusions, that is the process of the number of particles of a one-dimensional branching Brownian motion stopped upon exiting the interval  $((-\infty, {s}), {s}\geq 0)$. They discover in the supercritical case \cite{kaj1993note} that the resulting offspring distribution is degenerate, meaning that 
\be\label{eq-degenerateoffspring}
\sum_{i\geq 0} p_i < 1,
\e
where $p_i$ is the probability of having $i$ offspring, $i\geq 0$. In particular, the probability of a birth event with an infinite number of offspring is strictly positive. In this view, (\ref{eq-degenerateoffspring}) is the analogue of $\Psi_\infty(0)<0$. 

In Sheu \cite{sheu94, sheu97}, asymptotics of the process $Z$ are studied in order to obtain a compact support criterion for the super-Brownian motion $X$. It is found that the event of extinction of $Z$, i.e. $\{\exists s>0 : Z_s=0 \}$, and the event $\{X \text{ has compact support}\}$ 
agree $\mbf{P}_\mu$-a.s., c.f.  \cite{sheu97}, Theorem 4.1. \\
The following result on the asymptotic behaviour of $Z$ is given by Sheu \citep{sheu94}. 

\begin{theo*}[Sheu \cite{sheu94} Theorem 1.1, Theorem 1.2, Cor. 1.1] Let $\mu\in\mcal{M}_F(\mbb{R}^d)$. The event $\{\exists s>0 :Z_s=0 \}$ agrees $\mbf{P}_\mu$-a.s. with the event $\{\lim_{s\to\infty}Z_s = 0\}$  
if $\psi$ satisfies
\be\label{eq-sheuscondition2}
\int^\infty \frac{1}{ \sqrt{\int_{\lambda^*}^\lambda \psi(\theta) \ {\rm d}\theta} } \ {\rm d}\lambda < \infty.
\e
Otherwise,  $\{ \exists s>0 : Z_s=0 \}$  has probability $0$. 
\end{theo*}

In short, the event of extinction of $Z$ 
agrees with the event of extinguishing of $Z$, denoted by $\mcal{E}(Z):= \{\lim_{{s}\to\infty} Z_{s} = 0\}$, if and only if (\ref{eq-sheuscondition2}) holds, and it has zero probability otherwise. We have stated the theorem slightly differently from its original version in which, in the supercritical case, condition (\ref{eq-sheuscondition2})  reads  $ \int_s^\infty \frac{1}{\sqrt{{{\int_0^\lambda \phi(\theta) \ {\rm d}\theta}}}} \ {\rm d}\lambda < \infty,$ for $\phi(s):=\psi(s)-\alpha s$.
The equivalence of these two conditions was already  pointed out  in \cite{kyprianouetalsbmtw12}. \\
The unusual condition (\ref{eq-sheuscondition2}) corresponds to Grey's condition in \cite{greycsbps74} for extinction vs. extinguishing in the following sense. Recall that Grey's condition says that, for a standard CSBP with branching mechanism $F$, the event of extinction agrees with the event of becoming extinguished if and only if 
$\int^\infty F(\theta)^{-1} \ \ {\rm d}\theta < \infty$,
 and has probability zero otherwise. The following interpretation of (\ref{eq-sheuscondition2}) is an immediate consequence of Theorem \ref{theo-branchingmechlimit} (i). 

\begin{cor}\label{cor-sheugrey}
Sheu's compact support condition (\ref{eq-sheuscondition2}) is   Grey's condition for the limiting standard CSBP $Z^\infty$ with branching mechanism ${\Psi}_\infty$ in (\ref{eq-psibartheta}).
\end{cor}

Sheu's compact support condition (\ref{eq-sheuscondition2}) plays an important role when studying the radial speed of the support of supercritical Super-Brownian motion. In the one-dimensional case,  assuming (\ref{eq-sheuscondition2}), Kyprianou et. al \cite{kyprianouetalsbmtw12}, Corollary 3.2, show that 
\be\label{eq-stronglawradialspeed}
\lim_{t\to\infty} \frac{\mcal{R}_t}{ t} = \sqrt{-2 \psi'(0+)}, \ \ \ \mbf{P}_{\mu}-a.s, \ \mu\in\mcal{M}_F(\mbb{R}),
\e
where $\mcal{R}_t:=\sup\{r>0: X_t(r,\infty)>0\}$ is the right-most point of the support of $X_t$. A key step in the proof is to study  the  total mass of the process   of branching exit measures of a one-dimensional super-Brownian motion with drift $c := -\sqrt{-2 \psi'(0+)}$ upon exiting the increasing sequence of intervals $((-\infty,s),s\geq 0)$, which we denote here by $Z^c = (Z^{c}_s, s\geq 0)$. It is proved in Theorem 3.1 in \cite{kyprianouetalsbmtw12} that $Z^{c}$ is a subcritical standard CSBP. Now condition (\ref{eq-sheuscondition2}) comes in. Corollary \ref{cor-sheugrey} interprets (\ref{eq-sheuscondition2}) as  Grey's condition for the standard CSBP $Z^\infty$. The CSBPs $Z^\infty$ and $Z^c$ only differ in that the underlying Brownian motion of the latter  has drift $c$ and it is not difficult to convince ourselves that the drift term is irrelevant when studying the extinction vs. extinguishing problem, see (29) in \cite{kyprianouetalsbmtw12} for a rigorous argument. Therefore condition (\ref{eq-sheuscondition2}) is also equivalent to Grey's condition for the subcritical CSBP $Z^c$ and hence ensures that   $Z^c$ becomes extinct $\mbf{P}_\mu$-a.s. This now implies   that the right-most point of the support cannot travel at a speed faster than $\sqrt{-2 \psi'(0+)}$. In order to make this last conclusion, extinguishing of $Z^c$ is clearly not sufficient and 
it remains an open questions whether a strong law for $(\mcal{R}_t,t\geq 0)$ can exist when (\ref{eq-sheuscondition2}) fails. \\
In the $d$-dimensional case, $d\geq 1$, and with a quadratic branching mechanism of the form $\psi(\lambda)=- \alpha \lambda + \beta \lambda^2$, for $\alpha, \beta  \geq 0$, Kyprianou \cite{kyprianou2005asymptotic} shows that (\ref{eq-stronglawradialspeed}) holds, where $\mcal{R}_t$ is now replaced by $\tilde{\mcal{R}}_t:=\sup\{r>0: X_t(\mbb{R}^d\backslash D_r)>0\}$, the radius of the support of $X_t$. It can be checked that condition (\ref{eq-sheuscondition2})  is satisfied for this choice of $\psi$. It is possible to adapt the higher-dimensional result in \cite{kyprianou2005asymptotic} to hold for general branching mechanisms provided (\ref{eq-sheuscondition2}) holds. \\

The remainder of the paper is organised as follows. In Section \ref{sec-branchingmechpde} we prove Theorem \ref{theo-branchingmechpde} which is followed by the proof Proposition \ref{prop-changeshape} and Theorem \ref{theo-branchingmechlimit} in Section \ref{sec-branchingmechlim}.

%%%%%%%%%%%%%%%%%%%%%%%%%%%%%%%%%%%%%%%%%%%%%%%%%%%%%%%%%%%%%%%%%%%%%
%%%%%%%%%%%%%%%%%%%%%%%%%%%%%%%%%%%%%%%%%%%%%%%%%%%%%%%%%%%%%%%%%%%%%
%%%%%%%%%%%%%%%%%%%%%%%%%%%%%%%%%%%%%%%%%%%%%%%%%%%%%%%%%%%%%%%%%%%%
\section{Characterising the process $Z$ - Proof of Theorem \ref{theo-branchingmechpde}}\label{sec-branchingmechpde}

\subsection{Proof of Theorem \ref{theo-branchingmechpde} (i) and (ii)}\label{sec-proofoftheo1} 

\begin{proof}[Proof of Theorem \ref{theo-branchingmechpde} (i)]
 Take a look at equation (\ref{eq-strongmarkov}) which characterises the sequence of branching exit measures $(X_{D_s}, s\geq r)$. For any measure $\mu\in\mcal{M}_F(\partial D_r)$ and $||\mu||=a$, we can  write
\bex
E_{{a,r}}[e^{-\theta Z_{s}}]&=& \mbf{E}_{\mu}[e^{-\theta ||X_{D_{s}}||}]
= e^{-\la v_\theta (\cdot,{s}),\mu \ra}  
=e^{- v_\theta(x,{s}) a}, 
\ex
for any $x\in\p D_r$, by radial symmetry. The branching property of $Z$ now follows easily from the branching property of $(X_{D_{s}},{s}>r)$ in (\ref{eq-strongmarkov}) since, for $a,a'>0$, $0<r\leq {s}$, 
\bex
E_{(a+a') , {r}} [e^{-\theta Z_{s}}] 
&=& \mbf{E}_{\mu+\mu'}[e^{-\theta ||X_{D_{s}}||}] \\
&=& e^{- v_\theta(x,{s}) (a + a')} \\
&=& \mbf{E}_{\mu}[e^{-\theta ||X_{D_{s}}||}] \mbf{E}_{\mu'}[e^{-\theta ||X_{D_{s}}||}] 
= E_{a , {r}} [e^{-\theta Z_{s}}]E_{a' , {r}} [e^{-\theta Z_{s}}], 
\ex
for measures $\mu$, $\mu' \in \mcal{M}_F(\p D_r)$ with $||\mu||=a$, $||\mu'||=a'$.
The  Markov property is also an immediate consequence of (\ref{eq-strongmarkov}).
\end{proof}
\begin{proof}[Proof of Theorem \ref{theo-branchingmechpde} (ii)] 
First note that, by radial symmetry as seen in the proof of Theorem \ref{theo-branchingmechpde} (i), (\ref{eq-laplacedef}) holds with $u({r},{s},\theta)=v_\theta (x,{s})$ for $x\in\p D_r$ where $r=||x||$. It remains to show that  (\ref{eq-inhomlaplacefunctional}) and (\ref{eq-Psi}) are satisfied. \\
 For any $0<r\leq {z}\leq {s}$, $\theta\geq 0$,
\bex
E_r[e^{-\theta Z_{s}}] = E_{r} [ E_{Z_{z}, {z}}[e^{-\theta Z_{s}} ]] = E_{r} [e^{- u({z},{s},\theta) Z_{z}}] = e^{-u(r,{z},u({z},{s},\theta))},
\ex
which shows that the Laplace functional satisfies the composition property 
\be\label{eq-compositionprop}
u({r},{s},\theta) = u(r,{z}, u({z},{s},\theta)) \ \text{ for } 0<r\leq {z} \leq {s}, \  \theta\geq 0.
\e
The branching property of $Z$ implies that, for any fixed $0<r\leq{s}$,  the law of $(Z_{s},P_r)$ is an infinitely divisible distribution on $[0,\infty]$. It follows  from the L\'evy-Khintchin formula that, for fixed $r$ and $s$, $u({r},{s},\theta)$ is a non-negative, completely concave function as considered in Section 4 in Silverstein \cite{silverstein68}. The process $Z$ thus has the properties of the time-dependent version of the CSBP considered in Definition 4 in \cite{silverstein68}. 
We can then adapt the proof of Theorem 4 in \cite{silverstein68} to show that there exists a branching mechanism $\Psi$ of the form (\ref{eq-Psi}) such that
\bex
\frac{\p}{\p r} u({r},{s},\theta)\big|_{r=s}
= \Psi(s,\theta), \ \text{ for } {s}>0, \  \theta\geq 0.
\ex
With the composition property (\ref{eq-compositionprop}), we then get
\bex
\frac{\p}{\p r} u({r},{s},\theta) 
 &=&  \Psi(r,u({r},{s},\theta)), \ \text{ for } 0<r\leq {s}, \  \theta\geq 0.
\ex 
Indeed  it was already discussed at the end of Section 4 in \cite{silverstein68} that it is possible to allow time-dependence in Theorem 4 in \cite{silverstein68}. \\
Together with the initial condition $u(r,r,\theta) = \theta$, we obtain  equation (\ref{eq-inhomlaplacefunctional}).
\end{proof} 

From (\ref{eq-inhomlaplacefunctional}), we get an alternative characterisation of the relation between the Laplace functional $u$ and the branching mechanism $\Psi$ as
\be\label{eq-laplacefunctionalint}
\frac{\partial}{\partial {s}} u({r},{s},\theta) &=& - \Psi({s},\theta) \frac{\partial}{\partial \theta} u({r},{s},\theta) \\
\frac{\partial}{\partial r} u({r},{s},\theta) &=& \Psi(r,u({r},{s},\theta)) \label{eq-laplacefunctionalinr} \\
u(r,r,\theta)&=&\theta, \nonumber
\e
for any ${s}> r>0$ and  $\theta\geq 0$. To see where equation (\ref{eq-laplacefunctionalint}) comes from, compare the derivatives of (\ref{eq-inhomlaplacefunctional}) in ${s}$ and $\theta$, that is
\bex
\frac{\partial}{\partial {s}} u({r},{s},\theta) &=& - \Psi({s},\theta) - \int_{r}^{s} \frac{\p}{\p u}\Psi({z},u({z},{s},\theta)) \frac{\p}{\p {s}} u({z},{s},\theta) \ {\rm d}{z} \\
\frac{\partial}{\partial \theta} u({r},{s},\theta) &=& 1 - \int_{r}^{s} \frac{\p}{\p u}\Psi({z},u({z},{s},\theta)) \frac{\p}{\p \theta} u({z},{s},\theta) \ {\rm d}{z},
\ex
where ${\p}\Psi(\cdot,\cdot)/{\p u} $ denotes the derivative in the second component of $\Psi$. We see that $\frac{\partial}{\partial {s}} u({r},{s},\theta)$ and $- \Psi({s},\theta) \frac{\partial}{\partial \theta} u({r},{s},\theta)$ are solutions to the same integral equation. With an application of Gronwall's inequality it can be shown that this integral equation has a unique solution. 

\subsection{Proof of Theorem \ref{theo-branchingmechpde} (iii)} 
We have already seen in the previous section that, for any measure $\mu\in\mcal{M}_F(\partial D_r)$ with  $||\mu||=a$, we can  write
\bex
E_{{a,r}}[e^{-\theta Z_{s}}]&=& \mbf{E}_{\mu}[e^{-\theta ||X_{D_{s}}||}]
= e^{-\la v_\theta (\cdot,{s}),\mu \ra}  
=e^{- v_\theta(x,{s}) a},
\ex
for any $x\in\p D_r$, by radial symmetry. In particular,  we saw that $u({r},{s},\theta)=v_\theta (x,{s})$ for any $x\in\p D_r$.  From the semi-group equation for $v$ in (\ref{eq-semigroupexitmeasures}),
we thus get a semi-group representation of $u$, alternative to the representation in (\ref{eq-inhomlaplacefunctional}), as the unique non-negative solution to 
\be\label{eq-semigroupbessel}
u({r},{s},\theta) = \theta - \mbb{E}_r^{\text{R}} \Big[ \int_0^{\tau_{s}} \psi(u (R_{z},{s},\theta)) \ {\rm d}{z} \Big],
\e
where $(R,\mbb{P}_r^\text{R})$ is a $d$-dimensional Bessel process and $\tau_{s}:=\inf\{{z}>0: R_{z}>{s}\}$ its first passage time above level ${s}$.\\
Equation (\ref{eq-semigroupbessel}) tells us that the process  $Z$ can  be viewed as   the total mass process of the branching exit measures  of a $d$-dimensional super-Bessel process with branching mechanism $\psi$ as it first exits the  intervals $(0,{s})$, $s\geq r$.\\ 
Equivalently to the characterisation of $u({r},{s},\theta)$ as the unique non-negative solution to the integral equation (\ref{eq-semigroupbessel}), we can characterise it as the unique non-negative solution to the differential equation
\be\label{eq-pdebessel1}
\frac{1}{2}\frac{\partial^2}{\partial r^2} u({r},{s},\theta) + \frac{d-1 }{2 r} \frac{\partial}{\partial r} u({r},{s},\theta) &=&  \psi(u({r},{s},\theta)), \ \ \ 0<{r}< {s}, \theta\geq 0,\nonumber \\
u(r,r,\theta)&=&\theta. 
\e
We will show this equivalence in  \ref{app-ode}. In the following section, we will use the  differential equation (\ref{eq-pdebessel1}) to prove the PDE characterisation of the branching mechanism $\Psi$ in  Theorem \ref{theo-branchingmechpde} (iii). \\
We prove  Theorem \ref{theo-branchingmechpde} (iii)  in two parts. In Lemma \ref{lem-branchingmechpdepre} we show that $\Psi$ satisfies the PDE in (\ref{eq-pderelationforpsi}) before we prove that $\Psi(r,\lambda^*)=0$, for all $r>0$, in Lemma \ref{lem-rootlambdastar} below.
\begin{lem}\label{lem-branchingmechpdepre} 
The branching mechanism $\Psi$ satisfies the PDE (\ref{eq-pderelationforpsi}), i.e. 
\bex
\frac{\partial}{\partial r} \Psi(r,\theta) + \frac{1}{2}\frac{\partial}{\partial \theta} \Psi^2(r,\theta) + \frac{d-1}{r} \Psi(r,\theta) &=& 2 \psi(\theta) \ \ \ \ r>0, \ \theta\in (0,\infty).  
\ex
\end{lem}
\begin{proof}[Proof of Lemma \ref{lem-branchingmechpdepre}]
Using (\ref{eq-laplacefunctionalinr}), the left-hand side of (\ref{eq-pdebessel1}) becomes
\bex
\lefteqn{ \frac{\partial^2}{\partial r^2} u({r},{s},\theta) + \frac{d-1}{r} \frac{\partial}{\partial r} u({r},{s},\theta)} && \\
 &=& \frac{\partial}{\partial r} \Psi(r,u({r},{s},\theta)) + \frac{d-1}{r} \Psi(r,u({r},{s},\theta)) \\ &=& \frac{\partial}{\partial y}\Psi(y,u({r},{s},\theta))|_{y=r} \\
 && \qquad\qquad + \frac{\partial}{\partial u} \Psi(r,u({r},{s},\theta)) \ \Psi(r,u({r},{s},\theta))  + \frac{d-1}{r} \Psi(r,u({r},{s},\theta)) \\
&=& \frac{\partial}{\partial y}\Psi(y,u({r},{s},\theta))|_{y=r} + \frac{1}{2}\frac{\partial}{\partial u} \Psi^2(r,u({r},{s},\theta)) + \frac{d-1}{r} \Psi(r,u({r},{s},\theta)),
\ex
where ${\p \Psi(\cdot,\cdot)}/{\p u}$ denotes the derivative with respect to the second argument.
Note that this equation holds for all ${s}>{r}$ and $\theta\geq 0$. 
Since $u(r,s,\theta)\to\theta$ as $s\downarrow r$, we see that,  for fixed $r$, the range of $u({r},{s},\theta)$ is $(0,\infty)$ as we  vary ${s}\in(r,\infty)$ and $\theta\in [0,\infty)$. Hence, we can replace $u({r},{s},\theta)$ above by an arbitrary $\theta\in(0,\infty)$ and conclude that the PDE (\ref{eq-pderelationforpsi}) holds true.  
 \end{proof}

Recall that $\lambda^*=\inf\{\lambda\geq 0 : \psi(\lambda)>0 \}$ denotes the root of $\psi$ and define  $\lambda^*(r):=\inf\{\lambda\geq 0 : \Psi(r,\lambda)>0 \}$, for  $r>0$. 
\begin{lem}\label{lem-rootlambdastar}
(i) In the (sub)critical case,  for all $r > 0$, we have  $\lambda^*(r)= 0$. In particular, $\Psi(r,\theta)\geq 0$ for all $\theta\geq 0$.  \\
(ii) In the supercritical case, for all  $r>0$, we have 
$ \lambda^*(r)= \lambda^*$.
In particular, $\Psi(r,\theta)\leq 0$ for $\theta \leq \lambda^*$, while $\Psi(r,\theta) \geq 0$ for $\theta\geq \lambda^*$.
\end{lem}
\begin{proof}[Proof of Lemma \ref{lem-rootlambdastar} (i)]
As we are in the (sub)critical case we have $\psi(\theta)\geq 0$ for all $\theta\geq 0$. For $r<{z}<{s}$, (\ref{eq-semigroupbessel}) yields
\bex
u({r},{s},\theta) &=& \theta - \mbb{E}_{r}^{\text{R}}\int_0^{\tau_{s}} \psi(u(R_{v},{s},\theta)) \ {\rm d}v  \\
&=&  \theta - \mbb{E}_{r}^{\text{R}} \int_0^{\tau_{z}} \psi(u(R_{v},{s},\theta)) \ {\rm d}v  - \mbb{E}_{z}^{\text{R}}\int_0^{\tau_{s}} \psi(u(R_{v},{s},\theta)) \ {\rm d}v  \\
&\leq & \theta - \mbb{E}_{z}^{\text{R}}\int_0^{\tau_{s}} \psi(u(R_{v},{s},\theta)) \ {\rm d}v  \\
&=& u({z},{s},\theta).
\ex
Hence, $u({r},{s},\theta)$ is non-decreasing in $r$. With (\ref{eq-laplacefunctionalinr}) we thus see that,  for all $0<r<{s}$, $\theta \geq 0$, 
\be\label{eq-increasinginr}
\Psi(r,u({r},{s},\theta))=\frac{\partial}{\partial r} u({r},{s},\theta) \geq 0.
\e 
As we take $s\downarrow r$, we get $u(r,s,\theta)\to\theta$ and  hence  $\Psi(r,\theta)\geq 0$ for all $\theta>  0$, $r>0$. Continuity of $\Psi$ ensures $\Psi(r,0)=0$ and, in particular,  $\lambda^*({r})=0$ for all ${r}>0$. 
\end{proof}

The key to the proof of part (ii) of Lemma \ref{lem-rootlambdastar}  is the following lemma.
\begin{lem}\label{lem-martingales} Fix $r>0$. \\
(i)  
 For any $\lambda> 0$, the process 
\be\label{eq-martingalemt}
M_t^\lambda &=& e^{-\lambda Z_{s}} - \int_r^{s} \Psi({v},\lambda) Z_{v} e^{-\lambda Z_{v}} \mbf{1}_{\{Z_{v}<\infty\}}{\rm d}{v}, \ \ {s}\geq r,
\e
is a $P_r$-martingale. \\
(ii) The process 
$
(e^{-\lambda^* Z_{s}},  {s}\geq r)
$ 
 is a $P_r$-martingale.\\
 Here we use the convention $e^{-\lambda Z_{s}} \mbf{1}_{\{Z_{s}=\infty\}}=0$, 
 for any $\lambda>0$.
\end{lem}

\begin{proof}[Proof of part (i)]
Taking expectations in (\ref{eq-martingalemt}) and interchanging expectation and integral gives 
\bex
E_r[M_{s}^{\lambda}] &=& e^{-u({r},{s},\lambda)} - \int_r^{s} \Psi({v},\lambda) \frac{\partial}{\partial \lambda} u({r},{v},\lambda) \  e^{-u({r},{v},\lambda)} {\rm d}{v}.
\ex
Differentiating in ${s}$, together with (\ref{eq-laplacefunctionalint}), gives
\bex
\frac{\partial}{\partial {s}} E_r[M_{s}^{\lambda}] &=& \Big(- \frac{\partial}{\partial {s}} u({r},{s},\lambda) - \Psi({s},\lambda) \frac{\partial}{\partial \lambda} u({r},{s},\lambda) \Big) e^{-u({r},{s},\lambda)} = 0.
\ex
Hence, $E_r[M_{s}^\lambda]$ is constant for all ${s}\geq r$ and in particular, taking ${s}=r$, equal to $e^{-\lambda}$. Note that the same computation gives that $E_{a , {v}}[M_{s}^\lambda] = e^{-\lambda a}$, for $a>0$ and $0<r\leq {v}\leq {s}$. An application of the Markov property then shows that $(M^\lambda_{s}, {s}\geq r)$ is a martingale for any $\lambda > 0$. 
\end{proof}

The proof of Lemma \ref{lem-martingales} (ii) relies on the following idea. Since $(||X_{t}||, {t}\geq 0)$ is a CSBP with branching mechanism $\psi$ it is well-known that the process $(e^{-\lambda^*||X_{t}||}, {t}\geq 0)$ is a martingale with respect to the filtration $(\mcal{F}_{t},{t}\geq 0)$ where $\mcal{F}_{t}= \sigma(||X_{u}||,  {u}\leq {t})$. The martingale property follows on account of the fact that 
\bex
\mbf{E}_\mu[\mbf{1}_{\{||X_{u}||\to 0\}} | \mcal{F}_{{t}}]
=e^{-\lambda^* ||X_{{t}}||},  \ {t} \geq 0,
\ex 
by a simple application of  the tower property. Now, fix $r>0$, and consider  the filtration $(\mcal{G}_{{s}},{s}\geq {r})$ where $\mcal{G}_{{s}}=\sigma(||X_{D_{v}}||,{r}\leq {v}\leq {s}) = \sigma(Z_{v},{r}\leq{v}\leq {s})$ instead. If we can show that, for $\mu\in\mcal{M}_F(\p D_{r})$,
\bex
\mbf{E}_\mu[\mbf{1}_{\{||X_{u}||\to 0\}} | \mcal{G}_{s}]
=e^{-\lambda^* ||X_{D_s}||} = e^{-\lambda^* Z_s}, 
\ex
holds, then we can deduce in the same way that the process $(e^{-\lambda^*||X_{D_{s}}||} , {s}\geq {r})$ is a martingale with respect to the filtration $(\mcal{G}_{{s}}, {s}\geq {r})$. The proof is slightly cumbersome and therefore  postponed  to the end of this section. 

The proof of Lemma  \ref{lem-rootlambdastar} (ii) is now a simple consequence of Lemma \ref{lem-martingales}.
\begin{proof}[Proof of Lemma \ref{lem-rootlambdastar} (ii)]
By Lemma \ref{lem-martingales}, the process
\bex
e^{-\lambda^* Z_{s}} - M_t^{\lambda^*} = \int_r^{s} \Psi({v},\lambda^*) Z_{v} e^{-\lambda^* Z_{v}} \mbf{1}_{\{Z_{v}<\infty \}} \ {\rm d}{v}, \ \ {s}\geq r, 
\ex
must be a $P_r$-martingale. However this is only possible if the expectation of the Lebesgue-integral above is constant
in ${s}$ which requires  $\Psi({s},\lambda^*)=0$ on $\{0<Z_{s}<\infty\}$ for all ${s}\geq r$. Since the event $\{0<Z_{s}<\infty\}$ has positive probability  under $P_r$, we reason that $\Psi({s},\lambda^*)=0$ for all  ${s}\geq r$. Choosing $r>0$ arbitrarily small yields $\Psi({s},\lambda^*)=0$ for all ${s}>0$. Convexity of $\Psi({s},\theta)$ immediately implies that $\Psi({s},\theta)\geq 0$ for $\theta\geq \lambda^*$ and, further noting that $\Psi({s},0)\leq 0$,  that $\Psi({s},\theta)\leq 0$ for $\theta \leq \lambda^*$.
\end{proof}

\begin{proof}[Proof  of Theorem \ref{theo-branchingmechpde} (iii)]
Combine Lemma \ref{lem-branchingmechpdepre} and \ref{lem-rootlambdastar}. 
\end{proof}

Let us now come to the proof of Lemma \ref{lem-martingales} (ii). For $r>0$, $t\geq 0$, define the space-time domain $D_r^{t}$ as
\bex
D_r^{t} &=& \{(x,{u}): ||x||< r, {u}< {t}\} \subset \mbb{R}^d\times [0,\infty). 
\ex
Let $(X_{D_r^{t}}, {t}\geq 0, r>0)$ be the system of branching Markov exit measures describing the mass of $X$ as it first exits the space-time domains $D_{r}^{t}$,  see again Dynkin \cite{dynkinbemsp01}. \\ 
For the proof of Lemma \ref{lem-martingales} (ii), we will need the following  result which seems rather obvious but nevertheless needs a careful proof.
\begin{lem}\label{lem-aux} Let $r>0$.  
For any $\mu\in\mcal{M}_F(D_r)$, we have $\mbf{P}_\mu$-a.s.,
\bex
\lim_{{t}\to\infty} ||X_{D_r^{{t}}}|| &=& 
||X_{D_r}|| =Z_r. 
\ex 
\end{lem}
 
\begin{proof}
For $r>0$, $t\geq 0$, denote by $\partial D_r^{t}$ the boundary of the set $D_r^{t}$, i.e. 
\bex
\partial D_r^{t} &=& ( \ \{x: ||x||=r \} \times [0,{t})\ ) \cup (\ \{x: ||x||< r\} \times \{{t}\} \ ) \\
&=:& \partial D_r^{{t}-} \cup \partial D_{r-}^{t}.
\ex
By monotonicity, we have $\lim_{t\to\infty} \big| \big|  X_{D_r^{t}}\big|_{\p D_{r}^{{t}-}} \big| \big|  = \big| \big|  X_{D_{r}}\big| \big| = Z_r$, $\mbf{P}_\mu$-a.s. Next, define the event that $X$ becomes extinguished within $D_r$, i.e. 
\bex
\mcal{E}(X,D_r) &:=& 
 \{ \lim_{{t}\to\infty} \big|\big| X_{D_r^{t}}\big|_{\p D_{{r}-}^{t}} \big|\big|  =0 \}. \ex 
On the complement of $\mcal{E}(X,D_r)$, we have \bex
\lim_{t\to\infty} \big|\big| X_{D_r^{t}}\big|_{\p D_{{r}-}^{t}} \big|\big|=\infty, \ \  \mbf{P}_\mu-a.s.
\ex 
This is to say that, on  $\mcal{E}(X,D_r)^c$, the total mass within the open ball $D_r$ at time $t$ tends to infinity as  $t$ tends to infinity. This follows from Proposition 7 in \cite{englaenderkyprianou04} which says that 
$\limsup_{t\to\infty} ||X_{D_{r}^{t}}\big|_{B\times\{{t}\}}|| \in \{0,\infty\}$,
 $\mbf{P}_\mu$-a.s. for any nonempty open set $B\subset D_r$ (noting
that Proposition 7 in \cite{englaenderkyprianou04} indeed holds for the general branching mechanism
we are considering here). 
Hence, we have shown so far that
\bex
\lim_{{t}\to\infty} ||X_{D_r^{{t}}}|| &=& 
Z_r + \infty \mbf{1}_{\mcal{E}(X,{D_r})^c}.
\ex
Thus it remains to prove that, on $\mcal{E}(X,{D_r})^c$, $Z_r$ is also infinite. Fix a $K>0$. Thanks to Proposition 7 of \cite{englaenderkyprianou04}, on $\mcal{E}(X,{D_r})^c$, we can define an infinite sequence of stopping times
\bex
T_0 &=& \inf\{t>0 : \big|\big|X_{D_r^{t}}\big|_{\p D_{r-}^{{t}}}\big|\big|\geq K \} \\
T_{i+1} &=& \inf\{t>T_i+1: \big|\big|X_{D_r^{t}}\big|_{\p {D_{r-}^{t}}}\big|\big| \geq K\}, \ i=1,2,...
\ex
At times $T_i$, $i\geq 0$, the total mass within the open ball $D_r $ is greater than or equal to $K$. 
Fix an $M>0$ and define the event
\bex
A_i= \{\big|\big| X_{D_r^{T_i}}\big|_{[T_{i-1},T_i)\times \partial D_r}\big|\big| > M\}, \ \ \ i=1,2,...
\ex
which is the event that the mass that exits $D_r$ during the time interval $[T_{i-1},T_i)$ exceeds $M$. Note that there exists a strictly positive constant $\epsilon(M,K)$, such that
\be\label{eq-boundonprobs}
\mbf{P}_{X_{D_{r}^{T_i}}}(A_{i+1})& \geq & \mbf{P}_{K\delta_0}(A_1)  \nonumber \\
&\geq & \mbf{P}_{K\delta_0}(\big|\big| X_{D_r^1}\big|_{[0,1)\times \p D_r}\big|\big| >M) >\epsilon(M,K).
\e
Thus, we can partition time into infinitely many intervals $[T_i,T_{i+1})$, $i\geq 0$, of length at least $1$. During each time interval the mass that exits $D_r$, and thus contributes to $Z_r$, exceeds $M$ with positive probability. These probabilities are uniformly bounded from below by $\epsilon(M,K)>0$ in (\ref{eq-boundonprobs}). Therefore $||X_{D_r}|| = Z_r=\infty$, $\mbf{P}_\mu$-a.s on the event  $\mcal{E}(X,{D_r})^c$. This completes the proof.
\end{proof}

\begin{proof}[Proof of Lemma \ref{lem-martingales} (ii)]
For $s>0$, $t\geq 0$, define 
$\mcal{F}_{D_{s}^{t}} = \sigma(X_{D_{s'}^{{t}'}}, s'\leq s,  {t}'\leq {t})$.  
Fix $r>0$. The characterising branching Markov property for exit measures, see for instance Section 1.1 in \cite{dynkinkuznetsovnmeasures04}, yields that, for $\mu\in\mcal{M}_F(D_{r})$, ${s}\geq {r}$ and ${u}\geq {t} \geq 0$, we have
\be\label{eq-laplaceprojection2}
\mbf{E}_\mu[e^{-\theta ||X_{u}||} |\mcal{F}_{D_{s}^{t}}] = \exp\{-\la {w}_\theta(u-\cdot) , X_{D_{{s}}^{t}} \ra \}.
\e
where $w_\theta$ is the Laplace functional of the standard CSBP $(||X_{u}||, {u}\geq 0)$ with branching mechanism $\psi$. 
Taking $\theta=\lambda^*$, it is well known that $w_{\lambda^*}({t}) = \lambda^*$ for all $t\geq 0$. Therefore (\ref{eq-laplaceprojection2}), with $\theta$ replaced by $\lambda^*$, turns into
\bex
\mbf{E}_\mu[e^{-\lambda^* ||X_{u}||} |\mcal{F}_{D_{s}^{t}}] &=& \exp\{-\int w_{\lambda^*}({{u}-{t'}}) \ {\rm d}X_{D_{{s}}^{t}}({x},{t'})\} 
= e^{- \lambda^* ||X_{D_{s}^{t}}|| }.
\ex
Taking $u\to\infty$, we conclude
\be\label{eq-prelimitint}
\mbf{E}_\mu[\mbf{1}_{\{||X_{u}||\to 0\}} |\mcal{F}_{D_{s}^{t}}] = \lim_{u\to\infty} \mbf{E}_\mu[e^{-\lambda^* ||X_{u}||} |\mcal{F}_{D_{s}^{t}}] =  e^{- \lambda^* ||X_{D_{s}^{t}}|| }.
\e 
Now, we want to take the limit in $t$. By Lemma \ref{lem-aux}, we have $||X_{D_{s}^{t}}|| \to Z_{s}$ as $t\to\infty$ and thus the right-hand side of (\ref{eq-prelimitint}) tends to $\exp\{- \lambda^* Z_s \}$, $\mbf{P}_\mu$-a.s. For the left-hand side, by the strong Markov property, we can replace  $\mcal{F}_{D_{s}^{t}}$ by $\sigma(X_{D_{s}^{t}})$. Further, note that
$\mbf{P}_\mu(||X_{u}||\to 0)  = e^{-\lambda^* ||\mu||}$ for any $\mu\in\mcal{M}_F(D_s)$, with $\mbf{P}_\mu(||X_{u}||\to 0)  = 0$ if $\mu$ has infinite mass. Thus, the event $\{||X_{u}||\to 0\}$ only depends on the total mass of $\mu$. Therefore we can replace   $\sigma(X_{D_{s}^{t}})$ by $\sigma(||X_{D_{s}^{t}}||)$ on the left-hand side in (\ref{eq-prelimitint}). To sum up, we get
\bex
\mbf{E}_\mu[\mbf{1}_{\{||X_{u}||\to 0\}} |\mcal{F}_{D_{s}^{t}}] =\mbf{E}_\mu[\mbf{1}_{\{||X_{u}||\to 0\}} |\sigma(X_{D_{s}^{t}})] = \mbf{E}_\mu[\mbf{1}_{\{||X_{u}||\to 0\}} |\sigma(||X_{D_{s}^{t}}||)].
\ex
By Lemma \ref{lem-aux}, we have $\lim_{t\to\infty} ||X_{D_s^t}|| = Z_s$, with the possibility of the limit being infinite.  
Hence,  
\bex
\lim_{t\to\infty}\mbf{E}_\mu[\mbf{1}_{\{||X_{u}||\to 0\}} |\sigma(||X_{D_{s}^{t}}||)] =  \mbf{E}_\mu[\mbf{1}_{\{||X_{u}||\to 0\}} |\sigma(Z_{s})].
\ex 
Putting the pieces together, we get
\bex
\mbf{E}_\mu[\mbf{1}_{\{||X_{u}||\to 0\}} |\sigma(Z_{s})] 
= \lim_{t\to\infty}\mbf{E}_\mu[\mbf{1}_{\{||X_{u}||\to 0\}} |\mcal{F}_{D_{s}^{t}}] 
= \lim_{t\to\infty}  e^{- \lambda^* ||X_{D_{s}^{t}}|| }
= e^{- \lambda^* Z_s}.
\ex 
Finally take $\mu\in\mcal{M}_F(\p D_{r})$ and let $r\leq s' \leq s$. Then conditioning on $\sigma(Z_{{s}})$ and using the tower property, gives
\bex
 e^{-\lambda^* Z_{s'}} &=& \mbf{E}_\mu[\mbf{1}_{\{||X_{u}||\to 0\}} | \sigma(Z_{s'})] \\
 &=& \mbf{E}_\mu[\mbf{E}[\mbf{1}_{\{||X_{u}||\to 0\}} | \sigma(Z_{s}) ] | \sigma(Z_{s'})] 
={E}_{r}[ e^{-\lambda^* Z_s}| \sigma(Z_{s'})],
\ex
from which we conclude that $(e^{-\lambda^* Z_s}, s\geq {r})$ is a  $P_{r}$-martingale. 
\end{proof}
%%%%%%%%%% %%%%%%%%%%%%%%%%%%%%%%%%%%%%%%%%%%%%%%%%%%%%%%%%%%%%%%%%%
%%%%%%%%%%%%%%%%%%%%%%%%%%%%%%%%%%%%%%%%%%%%%%%%%%%%%%%%%%%%%%%%%
%%%%%%%%%%%%%%%%%%%%%%%%%%%%%%%%%%%%%%%%%%%%%%%%%%%%%%%%%%%%%%%%%%%

\section{The limiting branching mechanism - Proof of Proposition \ref{prop-changeshape} 
and Theorem \ref{theo-branchingmechlimit}} \label{sec-branchingmechlim}
\subsection{Changing shape - Proof of Proposition \ref{prop-changeshape}}
\begin{proof}[Proof of Proposition \ref{prop-changeshape}]
(i) Fix $0<r\leq r'$, $h>0$ and $\theta\geq 0$. The first step is to show that
$u({r}, {r}+h, \theta) \geq u({r}',{r}'+h,\theta)$.
Said  another way, we want to show that
\be\label{eq-inequality(sub)critical2}
E_{r'}[e^{-\theta Z_{r'+h}}] \geq   E_r[e^{-\theta Z_{r+h}}]. 
\e
Recall that $(Z_{r+h},P_r)$ is the total mass of $X$ as it first exists the ball $D_{r+h}$, when $X$ is initiated from one unit of mass distributed on $\p D_r$. By radial symmetry of $X$, we may assume that the initial mass is concentrated in a point $x_r\in\p D_r$, i.e. $ E_r[e^{-\theta Z_{r+h}}] =  \mbf{E}_{\delta_{x_r}}[e^{-\theta ||X_{D_{r+h}}||}]$. \\
Now we shift the point $x_r$ to the point $x_{r'}\in \p D_{r'}$ where $||x_{r'}-x_r|| = r'-r$. We also shift the ball $D_{r+h}$ in the same direction and by the same distance $r'-r$ and denote its new centre by $x_{r'-r}$, see Figure \ref{fig-shiftingballs}. 
\begin{figure}
\centering
\includegraphics[scale=0.4]{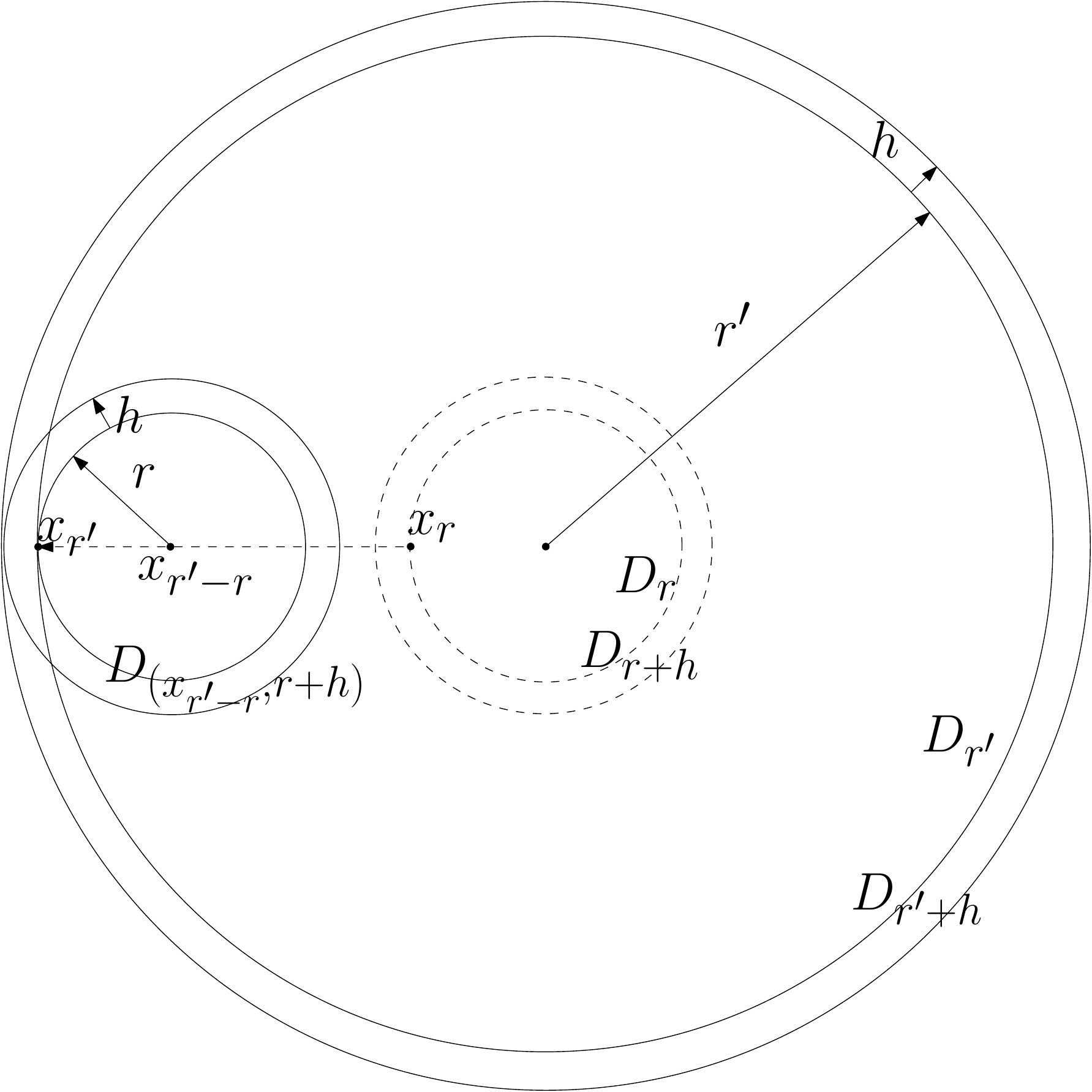}
\caption{Shifting the balls $D_r$ and $D_{r+h}$ by a distance $r'-r$}
 \label{fig-shiftingballs}
\end{figure}
By translation invariance of $X$ we then have
\bex
E_r[e^{-\theta Z_{r+h}}] = \mbf{E}_{\delta_{x_r}}\Big[e^{-\theta||X_{D_{r+h}}||}\Big] = \mbf{E}_{\delta_{x_{r'}}}\Big[ e^{-\theta||X_{D(x_{r'-r},r+h)}||}\Big],
\ex
where  $D(x_{r'-r},r+h)$ is the open ball centred at $x_{r'-r}$ with radius $r+h$.
We can then write (\ref{eq-inequality(sub)critical2}) as
\be\label{eq-inequality(sub)critical3}
\mbf{E}_{\delta_{x_{r'}}} \Big[e^{-\theta ||X_{D_{r'+h}}||}\Big] &\geq& \mbf{E}_{\delta_{x_{r'}}}\Big[ e^{-\theta||X_{D(x_{r'-r},r+h)}||}\Big].
\e
Recall that equation (\ref{eq-strongmarkov}) shows that the process of branching exit measure $X_{D_{{s}}}$ indexed by the increasing sequence of balls $(D_{{s}},{s}\geq {r})$ has the strong Markov property. By Dynkin \cite{dynkinbemsp01}, the strong Markov property holds more generally for any increasing sequence of open Borel subsets of $\mbb{R}^d$. In particular,
\be\label{eq-strongmarkovgeneral}
\mbf{E}_{\delta_{x_{r'}}} \Big[e^{-\theta ||X_{D_{r'+h}}||}\Big| \mcal{F}_{{D(x_{r'-r},r+h)}} \Big] =  \mbf{E}_{X_{D(x_{r'-r},r+h)}} \Big[e^{-\theta ||X_{D_{r'+h}}||}\Big],
\e\ 
where $\mcal{F}_{{D(x_{r'-r},r+h)}} = \sigma(X_{D(x_{r'-r}, s)}, s\leq r+h)$.
Hence, assuming that
\be\label{eq-usingsmp}
\mbf{E}_{X_{D(x_{r'-r},r+h)}}\Big[e^{-\theta ||X_{D_{r'+h}}||}\Big] \geq e^{-\theta ||X_{D(x_{r'-r},r+h)}||}
\e
holds true, we get, together with (\ref{eq-strongmarkovgeneral}), that
\bex
\mbf{E}_{\delta_{x_{r'}}} \Big[e^{-\theta ||X_{D_{r'+h}}||}\Big] 
&=& \mbf{E}_{\delta_{x_{r'}}} \Big[\mbf{E}_{\delta_{x_{r'}}} \big[e^{-\theta ||X_{D_{r'+h}}||}\big| \sigma(X_{D(x_{r'-r},r+h)}) \big]\Big] \\
&=&  \mbf{E}_{\delta_{x_{r'}}} \Big[\mbf{E}_{X_{D(x_{r'-r},r+h)}} \big[e^{-\theta ||X_{D_{r'+h}}||}\big]\Big] \\
&\geq& \mbf{E}_{\delta_{x_{r'}}} \Big[ e^{-\theta ||X_{D(x_{r'-r},r+h)}||}\Big], 
\ex
which is the desired inequality (\ref{eq-inequality(sub)critical3}). 
Thanks to the branching Markov  property for exit measures, for (\ref{eq-usingsmp}) to hold it suffices to show that 
\be\label{eq-inequality(sub)critical4}
\mbf{E}_{\delta_{x}}\Big[e^{-\theta ||X_{D_{r'+h}}||}\Big]  \geq e^{-\theta}, \ \ 
\text{ for any }  x\in\partial D{(x_{r'-r},r+h)}.
\e 
For fixed $x\in \partial D{(x_{r'-r},r+h)}$, set $s=||x||$ and note that $s\leq r'+h$.  By (\ref{eq-increasinginr}),  
$u({s},{r}'+h,\theta)$ is increasing in $s$ and bounded from above  by $u(r'+h,r'+h,\theta)=\theta$. Hence we obtain 
\bex
\mbf{E}_{\delta_{x}}[e^{-\theta ||X_{D_{r'+h}}||}] = E_{s}[e^{-\theta Z_{r'+h}}] = e^{-u({s},{r}'+h,\theta)} \geq e^{-\theta},
\ex
 which  is  (\ref{eq-inequality(sub)critical4}). This means we have proved (\ref{eq-inequality(sub)critical2}) and thus $u({r}, {r}+h, \theta) \geq u({r}',{r}'+h,\theta)$.
The latter yields 
that,  for all  $\theta\geq 0$,
\be\label{eq-monotonederiv}
\frac{\partial}{\partial {s}} u({r},{s},\theta)|_{{s}=r}  &=& \lim_{h\downarrow 0} \frac{u({r}, {r}+h, \theta)-u(r,r,\theta)}{h} \nonumber\\ 
&\geq& \lim_{h\downarrow 0 }\frac{u({r}',{r}'+h,\theta)-u(r',r',\theta)}{h} = \frac{\partial}{\partial {s}} u({r}',{s},\theta)|_{{s}=r'}. \nonumber \\
\e
Now we apply (\ref{eq-laplacefunctionalint}) to get
 \be\label{eq-laplacebranchingmech}
\frac{\partial}{\partial {s}} u({r},{s},\theta)|_{{s}=r} = \Big(-\Psi({s},\theta) \frac{\partial}{\partial \theta} u({r},{s},\theta)\Big)|_{{s}=r}= - \Psi(r,\theta) \cdot 1,
\e
where we used that $\lim_{{s}\downarrow r} \frac{\p}{\p\theta} u({r},{s},\theta)  = 1$ which can be seen as  follows. By dominated convergence, we have
\bex
\lim_{{s}\downarrow r} \frac{\p}{\p \theta} e^{-u({r},{s},\theta)} = \lim_{{s}\downarrow r} \frac{\p}{\p \theta} E_r[e^{-\theta Z_{s}} \mbf{1}_{\{Z_{s}<\infty\}}] = 
\lim_{{s}\downarrow r}  E_r[- Z_{s} e^{-\theta Z_{s}} \mbf{1}_{\{Z_{s}<\infty\}}] = - e^{-\theta}.
\ex
On the other hand,  
\bex
\lim_{{s}\downarrow r} \frac{\p}{\p \theta} e^{-u({r},{s},\theta)} = - \lim_{{s}\downarrow r} \frac{\p} {\p \theta} u({r},{s},\theta) \ e^{-u({r},{s},\theta)} = - \lim_{{s}\downarrow r} \frac{\p} {\p \theta} u({r},{s},\theta) \ e^{-\theta} 
\ex
 and we may conclude  that $\lim_{{s}\downarrow r} \frac{\p} {\p \theta} u({r},{s},\theta) = 1$ as claimed.\\
 Combining (\ref{eq-monotonederiv}) with (\ref{eq-laplacebranchingmech}) gives
$
\Psi(r,\theta) \leq \Psi(r',\theta)$ { for } $\theta \geq 0 
$ and $r\leq r'$, which completes the proof.

(ii) Define $\Psi^*(r,\theta):=\Psi(r,\lambda^*+\theta)$ for $\theta\geq 0$. Then $(\Psi^*(r,\cdot), r>0)$ is a family of subcritical branching mechanisms 
which, by part (i), has the property that $\Psi^*(r,\theta) \leq \Psi^*(r',\theta)$ for $r\leq r'$ and all $\theta \geq 0$. Clearly this gives  $\Psi(r,\theta) \leq \Psi(r',\theta)$ for $r\leq r'$ and $\theta \geq \lambda^*$. \\
Let $\theta\leq \lambda^*$. First, note that $u(r,s,\lambda^*)= - \log E_r[e^{-\lambda^* Z_s}] = \lambda^*$, which is a consequence of Lemma \ref{lem-martingales} (ii). Thus, $u(r,s,\theta) \leq u(r,s,\lambda^*)=\lambda^*$ for all $\theta\leq \lambda^*$,  $0<r\leq s$, and in particular  $\psi(u(r,s,\theta))\leq 0$. We therefore get 
\bex
u({r},{s},\theta) 
&=&  \theta - \mbb{E}_{r}^{\text{R}} \int_0^{\tau_{z}} \psi(u(R_{v},{s},\theta)) \ {\rm d}v  - \mbb{E}_{z}^{\text{R}}\int_0^{\tau_{s}} \psi(u(R_{v},{s},\theta)) \ {\rm d}v \nonumber \\
&\geq & \theta - \mbb{E}_{z}^{\text{R}}\int_0^{\tau_{s}} \psi(u(R_{v},{s},\theta)) \ {\rm d}v \nonumber \\
&=& u({z},{s},\theta)
\ex
for any $0<r\leq z \leq s$, $\theta\leq \lambda^*$.
We can then use  $\frac{\p}{\p r} u({r},{s},\theta) \leq 0 $ in place of the inequality (\ref{eq-increasinginr}) in the proof of part (i). 
Thus, following the same arguments as in the proof of part (i) with all inequalities reversed, we  see that  $\Psi(r,\theta) \geq \Psi(r',\theta)$ for $r\leq r'$ and all $\theta \leq \lambda^*$.
\end{proof}

\subsection{Limiting branching mechanism - Proof of Theorem \ref{theo-branchingmechlimit}}
To begin with, we show the existence and finiteness of the limiting branching mechanism $\Psi_\infty$ and derive a PDE characterisation. 
\begin{prop}\label{prop-pdeforpsiinfinity}
For each $\theta\geq 0$, the limit $\lim_{r\uparrow \infty}\Psi(r,\theta)={\Psi}_\infty(\theta)$  is finite and the convergence holds uniformly in $\theta$ on any bounded, closed subset of $\mbb{R}_+$. 
\\
(i) In the (sub)critical case, 
${\Psi}_\infty$ 
satisfies the equation
\be\label{eq-limitbranchingmech}
\frac{1}{2} \frac{\partial}{\partial \theta} {\Psi}_\infty^2(\theta) &=& 2 \psi(\theta), \ \ \  \theta\geq 0, 
\\
{\Psi}_\infty(0) &=& 0.\nonumber
\e
(ii) In the supercritical case, ${\Psi}_\infty$ satisfies (\ref{eq-limitbranchingmech})  with the initial condition at $0$ replaced by
\bex
{\Psi}_\infty(0) &=& - 2 \sqrt{\int_0^{\lambda^*}|\psi(\theta)| \ {\rm d}\theta} 
\ex
and ${\Psi}_\infty(\lambda^*) = 0$.  
\end{prop}

\begin{proof}
 From the monotonicity in Proposition \ref{prop-changeshape}, we conclude that the pointwise limit ${\Psi}_\infty(\theta):= \lim_{r\uparrow \infty}\Psi(r,\theta)$ exists. We will have to show that  $|{\Psi}_\infty(\theta)|$ is finite for each $\theta\geq 0$. Uniform  convergence on any bounded, closed subset of $\mbb{R}$ will then follow by convexity, see for example Theorem 10.8 in \cite{rockafellarbookconvex97}.  We consider the (sub)critical case and the supercritical case separately.
 
 (i) Suppose we are in the (sub)critical case. We have $\Psi(r,0)=0$ for all $r>0$ and hence ${\Psi}_\infty(0)=0$.  For $\theta>0$, recall  the PDE (\ref{eq-pderelationforpsi}), which can be written slightly differently as
\be\label{eq-pderelationforpsi2}
\frac{\partial}{\partial r} \Psi(r,\theta) + \Psi(r,\theta)\frac{\partial}{\partial \theta} \Psi(r,\theta) + \frac{d-1}{r} \Psi(r,\theta) &=& 2 \psi(\theta), \ \ r>0, \theta> 0. \nonumber\\
\e
 By Proposition \ref{prop-changeshape} (i), $\frac{\p}{\p r} \Psi(r,\theta)\geq 0$ and, by Lemma \ref{lem-rootlambdastar}(i), $\Psi(r,\theta)\geq 0$. Thus, 
\be\label{eq-inequalityforpsi}
\Psi(r,\theta)\frac{\partial}{\partial \theta} \Psi(r,\theta) &\leq& 2 \psi(\theta), \ \ \ \text{ for all } r>0 \text{ and } \theta\geq 0. 
\e 
 Fix a $\theta_0> 0$. Suppose for contradiction that $\Psi(r,\theta_0)\uparrow \infty$ as $r\to\infty$. For any $K>0$, we can find an $r_0$ large enough such that \be\label{eq-assumptionunbounded}
\Psi(r_0,\theta_0)&>& 2 K \psi(\theta_0).
\e
By (\ref{eq-inequalityforpsi}), this implies that
$\frac{\p}{\p \theta} \Psi(r_0,\theta_0) < \frac{1}{K}$. 
As $\Psi$ is convex in $\theta$ with $\Psi(r_0,0)=0$, we get that
\bex
\Psi(r_0,\theta_0) &\leq& \frac{\theta_0}{K}.
\ex
Now we can choose $K$ large enough such that $\theta_0/K < 2 K \psi(\theta_0)$, which then contradicts (\ref{eq-assumptionunbounded}). 
Hence, $\lim_{r\to\infty}\Psi(r,\theta)=\Psi_\infty(\theta)<\infty$ for all $\theta\geq 0$. \\
Note that $\limsup_{r\to\infty} \frac{\p}{\p\theta} \Psi(r,\theta)$ is also finite for each $\theta\geq 0$.  Indeed, if we supposed the contrary for some $\theta > 0$, that is,  $\limsup_{r\to\infty}\frac{\partial}{\partial \theta} \Psi(r,\theta) = \infty$, then (\ref{eq-inequalityforpsi}) would imply that $\liminf_{r\to\infty}\Psi(r,\theta)= 0$, which contradicts Lemma \ref{lem-rootlambdastar} (i). By convexity, we can pick any $\theta > 0$ to get $\limsup_{r\to\infty} \frac{\partial}{\partial \theta}  \Psi(r,0+)  \leq \limsup_{r\to\infty} \frac{\partial}{\partial \theta} \Psi(r,\theta)<\infty$.
\\
Next, we want to take  $r\to\infty$ in (\ref{eq-pderelationforpsi2}) and we know that the limit of the left-hand side exists since the right-hand side does  not depend on $r$. We keep $\theta_0>0$ fixed and  
 consider each term on the left-hand side of (\ref{eq-pderelationforpsi2}) separately.\\
We have just seen that $\lim_{r\to\infty}\Psi(r,\theta_0)<\infty$ which implies that the third term on the left-hand side of  (\ref{eq-pderelationforpsi2}), namely $\frac{d-1}{r} \Psi(r,\theta_0)$, vanishes as $r\to\infty$. \\
Consider the term $\Psi(r,\theta_0) \frac{\partial}{\partial \theta} \Psi(r,\theta_0)$ next. Since $\Psi(r,\cdot)$ is a sequence of continuous, convex functions, the pointwise limit ${\Psi}_\infty$ is also continuous and convex in $\theta$, cf. Theorem 10.8 in Rockafellar \cite{rockafellarbookconvex97}.
The convexity  ensures that the set of points at which ${\Psi}_\infty$ is not differentiable is at most countable. If $\Psi_\infty$ is differentiable at $\theta_0$, then by Theorem 25.7 in \cite{rockafellarbookconvex97}, it follows that  $\lim_{r\to\infty} \frac{\p}{\p\theta} \Psi(r,\theta_0) = \frac{\p}{\p\theta} \Psi_\infty(\theta_0)$ 
and hence 
\be\label{eq-limitafterswap}
\lim_{r\to\infty}  \Psi(r,\theta_0)\frac{\partial}{\partial \theta} \Psi(r,\theta_0) = {\Psi}_\infty(\theta_0) \frac{\partial}{\partial \theta} {\Psi}_\infty(\theta_0).
\e
So far we have seen that, for all $\theta\geq 0$ at which $\Psi_\infty$ is differentiable, the second and third term on the left-hand side of (\ref{eq-pderelationforpsi2}) converge to a finite limit as $r\to\infty$ which implies that the limit of the first term, that is $\lim_{r\to\infty}\frac{\p}{\p r}\Psi(r,\theta)$, also exists and is finite. With  $\lim_{r\to\infty}\Psi(r,\theta)<\infty$ it thus follows that $\frac{\p}{\p r}\Psi(r,\theta)$ tends to $0$ as $r\to\infty$, for all $\theta\geq 0$ at which $\Psi_\infty$ is differentiable.\\
In conclusion, for any  $\theta$ at which $\Psi_\infty$ is differentiable,  the first and third term on the left-hand side of (\ref{eq-pderelationforpsi2}) vanish as $r\to\infty$ and with (\ref{eq-limitafterswap}) we get
\be\label{eq-limitforallmostall}
\Psi_\infty(\theta) \frac{\partial}{\partial \theta} \Psi_\infty(\theta) = 2 \psi(\theta).  
\e
For $\theta>0$, we have $\Psi_\infty(\theta)>0$ and we can write (\ref{eq-limitforallmostall}) as
\be\label{eq-limitforallmostall2}
 \frac{\partial}{\partial \theta} \Psi_\infty(\theta) = 2 \frac{\psi(\theta)}{\Psi_\infty(\theta)},
\e
which again holds for all $\theta>0$ at which  $\Psi_\infty$ is differentiable. By convexity,  $\Psi_\infty$ admits left and  right derivatives for every $\theta>0$. Since the right-hand side of (\ref{eq-limitforallmostall2}) is continuous and (\ref{eq-limitforallmostall2}) holds true for all but countably many $\theta>0$, we conclude that the  left and the right  derivative of $\Psi_\infty(\theta)$ agree for every $\theta > 0$. Thus (\ref{eq-limitforallmostall2}), and equivalently   (\ref{eq-limitbranchingmech}), holds in fact for every $\theta >0$.
By convexity, for any $\theta>0$, we get  
\bex
\frac{\partial}{\partial \theta} \Psi_\infty(0+)\leq  \frac{\partial}{\partial \theta} \Psi_\infty(\theta) = 2 \frac{\psi(\theta)}{\Psi_\infty(\theta)} <\infty,
\ex
which shows that (\ref{eq-limitbranchingmech}) holds true for $\theta = 0$ with both sides being equal to $0$.

(ii) We consider the supercritical case now. Again we first have to show that ${\Psi}_\infty(\theta)$ is finite for each $\theta\geq 0$. \\
Let us begin with the case $\theta\in[\lambda^*,\infty)$.  We can consider the (sub)critical branching mechanisms $\Psi^*(r,\lambda) := \Psi(r,\lambda+\lambda^*)$ for $\lambda\geq 0$. 
Then  part (i)  applies to the  (sub)critical $\Psi^{*}$ and we conclude that, for any $\theta\geq \lambda^*$,
\bex
{\Psi}_\infty(\theta) = \lim_{r\to\infty} \Psi(r,\theta) = \lim_{r\to\infty} \Psi^{*}(r,\theta-\lambda^*) = {\Psi}_\infty^{*}(\theta-\lambda^*)<\infty. 
\ex
In particular, the equation (\ref{eq-limitbranchingmech}) holds for all $\theta\geq\lambda^*$ and ${\Psi}_\infty(\lambda^*)={\Psi}_\infty^*(0)=0$.\\
Further, it follows from the monotonicity in Proposition \ref{prop-changeshape} that $\frac{\p}{\p \theta} \Psi^*(r,0+)\leq \frac{\p}{\p \theta} {\Psi}_\infty^*(0+)$. The latter derivative was shown to be finite in the proof of part (i).
Thus, for any $r>0$, 
\bex
\frac{\p}{\p\theta}\Psi(r,\theta)|_{\theta =\lambda^*} = \frac{\p}{\p \theta} \Psi^*(r,0+)\leq \frac{\p}{\p \theta} {\Psi}_\infty^*(0+)<\infty.
\ex
Hence, we have a uniform upper bound for the $\theta$-derivative of $\Psi(r,\cdot)$ at $\lambda^*$. Recalling that $\Psi(r,\lambda^*)=0$,
 convexity ensures that $\Psi(r,\cdot)$ is uniformly bounded from below by the function $\frac{\p}{\p \theta} {\Psi}_\infty^*(0+)(\cdot - \lambda^*)$. 
This implies already that  $\lim_{r\to\infty}|\Psi(r,\theta)|<\infty$ for all $\theta\in[0,\lambda^*]$. \\
To show that the equation (\ref{eq-limitbranchingmech}) holds for all $\theta\leq\lambda^*$ we can now simply repeat the argument given in the proof of part (i).
Finally, with ${\Psi}_\infty(\lambda^*)=0$, we can derive the initial condition for ${\Psi}_\infty(0)$ by integrating (\ref{eq-limitbranchingmech}) from $0$ to $\lambda^*$. 
\end{proof}

\begin{proof}[Proof of Theorem \ref{theo-branchingmechlimit}]
 Proposition \ref{prop-pdeforpsiinfinity} guarantees the existence and finiteness of $\Psi_\infty$. 
 If we integrate (\ref{eq-limitbranchingmech}) from $\lambda^*$ to $\theta$, and note that $\Psi_\infty(\theta)$ and $\psi(\theta)$ are negative if and only if $\theta\leq \lambda^*$, we obtain the  expression in (\ref{eq-psibartheta}). 
 It thus remains to show (ii). 

It follows from an obvious adaptation of the proof of Theorem 3.1 in Kyprianou et al. \cite{kyprianouetalsbmtw12} that $Z^\infty$ is the process of the total mass  of the branching Markov exit measures of a one-dimensional super-Brownian as it first exits the family of intervals $((-\infty, {s}), {s}\geq 0)$ as claimed. \\
Concerning the convergence in (\ref{eq-convergencetozinfty}), we will show that, for $s\geq 0$ and $\theta\geq 0$,  $u^\infty({s},\theta):=\lim_{r\to\infty} u({r},{s}+{r},\theta)$ exists and solves
\be\label{eq-semigrouplimit}
u^\infty({s},\theta) = \theta - \int_0^{s} {\Psi}_\infty(u^\infty({s}-{v},\theta)) \ {\rm d}{v},
\e
which is the characterising  equation for the Laplace functional of ${Z}^\infty$. \\
This is trivially satisfied for ${s}=0$. Henceforth, let  ${s}>0$ and $\theta\geq 0$ be fixed. 
Recall that $u({r},{s}+{r},\theta)$ solves equation  (\ref{eq-inhomlaplacefunctional}), which can be written as
\bex %\label{eq-inhomlaplacefunctional2}
u({r},{s}+{r},\theta) = \theta - \int_0^{s} \Psi({v}+r,u({v}+{r},{s}+{r},\theta)) \ {\rm d}{v}, \ \ \,  r> 0.
\ex
Note that the convergence of the convex functions $\Psi(r,\cdot)$ to $\Psi_\infty(\cdot)$ in Theorem \ref{theo-branchingmechlimit} holds uniformly in $\theta$ on each bounded closed subset of $\mbb{R}_+$. 
Therefore, for fixed $\epsilon>0$, we can 
choose  $r$ large enough such that 
$|\Psi({s}+{r},\lambda)-{\Psi}_\infty(\lambda)|<\epsilon$ for all $\lambda\in {\{u({v}+{r},{s}+{r},\theta), 0\leq {v}\leq {s}\}}$. 
Thus, for large $r$,
\be\label{eq-convergenceofu}
\lefteqn{ \Big|u({r},{s}+{r},\theta)- \big(\theta- \int_0^{s} {\Psi}_\infty(u({v}+{r},{s}+{r},\theta))\ {\rm d}{v} \big)\Big| } && \nonumber\\
&=& \Big|\int_0^{s}\Psi({v}+r,u({v}+{r},{s}+{r},\theta)) \ {\rm d}{v} -  \int_0^{s} {\Psi}_\infty(u({v}+{r},{s}+{r},\theta))\ {\rm d}v \Big| \nonumber\\ 
&\leq& \epsilon \ {s}.
\e
Now assume for  contradiction that $\limsup_{{r}\to\infty} u({r},{s}+{r},\theta) = +\infty$. Since $\Psi_\infty$ is convex and $\Psi'_\infty(0+)\geq 0$ (with $\Psi'_\infty(0+)=0$ in the supercritical case), the integrand in the first line of (\ref{eq-convergenceofu}) is bounded from below by $\Psi_{\infty}(0)$. Therefore, the expression in the first line of (\ref{eq-convergenceofu}) tends to $\infty$ along a subsequence of $r$ which is an obvious contradiction.\\
Hence, $u({r},{s}+{r},\theta)$ is  bounded as a sequence in $r$. It  therefore contains a convergent subsequence, say $u({r}_n,{s}+{r}_n,\theta)$ where $(r_n,n\geq 1)$ is a strictly monotone sequence which tends to $\infty$.  
Let us show that every subsequence converges to the same limit. 
Let $(r'_n, n\geq 1)$ be another strictly monotone sequence which tends to $\infty$. Set $u_{\sup}({v}) := \sup_{n\in\mbb{N}}\{u({r}_n,{v}+{r}_n,\theta)\}$ and  $u'_{\sup}({v}) := \sup_{n\in\mbb{N}}\{u({r}'_n,{v}+{r}'_n,\theta)\}$ and note that  $u_{\sup}({v})$, $u'_{\sup}({v}) <\infty$.
By (\ref{eq-convergenceofu}), for any $\epsilon>0$, we can find an $N\in\mbb{N}$ large enough such that for all $n\geq N$
\be\label{eq-estimateforsubsequences}
\lefteqn{ |u({r}_n,{s}+{r}_n,\theta) - u({r}'_n,{s}+{r}'_n,\theta) | } &&\nonumber\\
&\leq& 2\epsilon + \int_0^{s} \Big|\Psi_\infty(u({r}_n,{v}+{r}_n,\theta)) - {\Psi}_\infty(u({r}'_n,{v}+{r}'_n,\theta))\Big| \ {\rm d}{v}  \nonumber\\
&\leq&  2\epsilon +\int_0^{s} M({v}) |u({r}_n,{v}+{r}_n,\theta) - u({r}'_n,{v}+{r}'_n,\theta)|\ {\rm d}{v},
\e
where $M({v}):= \sup\{\Psi'_\infty(w) : w\in (0,   u_{\sup}({v}) \vee u'_{\sup}({v}) )\}<\infty$. We can bound the integral further by setting $M:=\sup_{0\leq {v} \leq s} M({v})<\infty$. Set \bex
F_n({s}') = M \int_0^{{s}'} |u({r}_n,{v}+{r}_n,\theta) - u({r}'_n,{v}+{r}'_n,\theta)|\ {\rm d}{v}, \ \ \text{ for } \ 0\leq {s}' \leq {s}, 
\ex 
 and note that $\p F_n({s}')/\p {s}' =  M |u({r}_n,{s}'+{r}_n,\theta) - u({r}'_n,{s}'+{r}'_n,\theta)|$. By (\ref{eq-estimateforsubsequences}),
\bex
\frac{\p}{\p {s}'} F_n({s}') - 2 \epsilon M - M F_n({s}') \leq 0. 
\ex  
Multiplying by $e^{-M{s}'}$, we derive
$
\p [(F_n({s}')+2\epsilon) e^{-M{s}'}]/\p s' \leq 0.
$
Therefore,  
\bex
(F_n({s}')+2\epsilon) e^{-M{s}'}  \leq  F_n({0})+2\epsilon = 2\epsilon, \ \ \text{ for any } \ 0\leq s' \leq s.
\ex
Hence, $F_n({s}') \leq 2\epsilon (e^{M{s}}-1)  $, for  $0\leq s' \leq s$.
Since $\epsilon>0$ can be chosen arbitrarily small, we  conclude from the definition of $F_n(s')$ that $u({r}'_n,{s}'+{r}'_n,\theta)$ converges to the same limit as $u({r}_n,{s}'+{r}_n,\theta)$ as $n\to\infty$.  
We have thus shown that, considered as a sequence in $r$, all subsequences of $u({r},{s}+{r},\theta)$ converge to the same limit. Therefore  $u^\infty({s},\theta) =\lim_{r\to\infty} u({r},{s}+{r},\theta)$ exists and, with (\ref{eq-convergenceofu}),  it satisfies  (\ref{eq-semigrouplimit}). 
By uniqueness of solutions to (\ref{eq-semigrouplimit}),   $u^\infty({s},\theta)$ agrees with the Laplace functional associated with ${Z}^\infty$ which in turn implies the desired convergence.
\end{proof}

\appendix
\section{Derivation of the differential equation (\ref{eq-pdebessel1}) corresponding to the semi-group equation (\ref{eq-semigroupbessel})}\label{app-ode} 
The reader familiar with the superprocess literature will readily believe that any solution to  the differential equation (\ref{eq-pdebessel1}) also solves the semi-group equation (\ref{eq-semigroupbessel}) and conversely that solutions to (\ref{eq-semigroupbessel}) also solve (\ref{eq-pdebessel1}). Results of this fashion can be found for instance in the work of Dynkin, see \cite{dynkin1991probabilistic},  Section 3 in \cite{dynkinsuperprocpde93} or  Section 5.2 in   \cite{dynkindiffsuperdiffpde02}. 
However, in these references  only (sub)critical branching mechanism are allowed and the authors are unaware of a rigorous proof in the literature for the case of a supercritical branching mechanism.
Although it seems possible to adapt Dynkin's arguments  to the supercritical case, we will offer a self-contained proof here instead. \\
\\
Recall that the Laplace functional $u$ of $Z$, defined in (\ref{eq-laplacedef}), is the unique non-negative solution to the equation  
\be\label{eq-semigroupbessel2}
u({r},{s},\theta) = \theta - \mbb{E}_r^{\text{R}} \int_0^{\tau_{s}} \psi(u (R_{l},{s},\theta)) \ {\rm d}{l}, \ 0<r\leq {s}, \ \theta\geq 0,
\e
where $(R,\mbb{P}^\text{R})$ is a $d$-dimensional Bessel process and $\tau_{s}:=\inf\{{l}>0: R_{l}>{s}\}$ its first passage time above level ${s}$, see (\ref{eq-semigroupbessel}).\\
Fix $0<r\leq {s}$ and $\theta\geq 0$ from now on.
Let us apply a Lamperti transform to the $d$-Bessel process $R$ in the integral on the right-hand side of (\ref{eq-semigroupbessel2}).  
Define $\varphi({s})=\int_0^{r^2 {s}} R_{l}^{-2} {\rm d}{l}$, ${s}\geq 0$, then
\bex
B_{s} = \log (r^{-1} R_{r^2 \varphi^{-1}({s})}), \ {s}\geq 0,
\ex
is a one-dimensional Brownian motion with drift $\frac{d}{2}-1$ starting from $0$. Let us denote the law of $B$ by $\mbb{P}_0$. Thus we get
\bex
\mbb{E}_r^{\text{R}} \int_0^{\tau_{s}} \psi(u (R_{l},{s},\theta)) \ {\rm d}{l} &=& \mbb{E}_r^{\text{R}} \int_0^{\varphi(r^{-2}\tau_{s})} \psi(u(R_{r^2\varphi^{-1}({l})},{s},\theta)) R_{r^2\varphi^{-1}({l})}^2\ {\rm d}{l} \\
&=& {\mbb{E}}_0 \int_0^{T_{ \log({s}/r) }} \psi(u(e^{B_{l}+\log r},{s},\theta))e^{2 (B_{l}+\log r)}\ {\rm d}{l} \\
&=& \mbb{E}_{\log r} \int_0^{T_{\log {s} }} \psi(u (e^{B_{l}},{s},\theta))e^{2 B_{l}}\ {\rm d}{l}, 
\ex
where $T_{\log {s}}$ is the first time $B$ crosses level $\log {s}$. 
Equation  (\ref{eq-semigroupbessel2}) becomes
\be\label{eq-semigroupbrownianmotion}
u({r},{s},\theta)
&=& \theta - \mbb{E}_{\log r} \int_0^{T_{ \log {s} }} \psi(u(e^{B_{l}},{s},\theta))e^{2 B_{l}}\ {\rm d}{l}.
\e
 We split the integral on the right hand side into its excursions away from the maximum. This gives
\bex
&&  \mbb{E}_{\log r} \int_0^{T_{ \log {s} }} \psi(u(e^{B_{l}},{s},\theta))e^{2 B_{l}}\ {\rm d}{l} \\
   && \qquad \qquad = \mbb{E}_{\log r} \sum_{\log r\leq u \leq \log {s}} \int_0^{\zeta^{(u)}} \psi(u(e^{u-e_u({l})},{s},\theta)) e^{2(u-e_u({l}))} \ {\rm d}{l},
\ex
where $e_u$ is an excursion away from the maximum with lifetime $\zeta^{(u)}$ and the sum is taken over all left end-points $u$ of the excursion intervals in $(T_{\log r}, T_{\log {s}})$. It follows from the Compensation formula for excursions (Bertoin \cite{bertoin1998booklevy},  Cor. 11, p.110) that 
\bex
&& \mbb{E}_{\log r} \sum_{\log r\leq u \leq \log {s}} \int_0^{\zeta^{(u)}} \psi(u(e^{u-e_u({l})},{s},\theta)) e^{2(u-e_u({l}))} \ {\rm d}{l} \\&& \qquad \qquad \qquad = \int_{\log r}^{\log {s}} \eta \left( \int_0^\zeta \psi(u( e^{u-e({l})},{s},\theta) ) e^{2(u-e({l}))} \ {\rm d}{l}\right) \ {\rm d}u,
\ex
where $\eta$ denotes the excursion measure and $e$ is a generic excursion with length $\zeta$. Then we apply Exercise 5, chapter VI, \cite{bertoin1998booklevy}, to get
\bex
&& \int_{\log r}^{\log {s}} \eta \left( \int_0^\zeta \psi(u( e^{u-e({s})},{s},\theta) ) e^{2(u-e({l}))} \ {\rm d}{l}\right) \ {\rm d}u \\
&& \qquad \qquad = \int_{\log r}^{\log {s}} \int_0^\infty \psi(u( e^{u-y},{s},\theta )) e^{2(u-y)}  \ \hat{V}({\rm d}y) \ {\rm d}u,
\ex
where $\hat{V}$ is the renewal function of the dual ladder height process (the dual process is here simply Brownian motion with drift $-(\frac{d}{2}-1)$). 
We see from equation (4), p. 196 in \cite{bertoin1998booklevy} that $\hat{V}({\rm d}y)= 2 e^{-2 (\frac{d}{2}-1) y} {\rm d}y$ and obtain
\bex
\lefteqn{ \int_{\log r}^{\log {s}} \int_0^\infty \psi(u ( e^{u-y},{s},\theta ) ) e^{2(u-y)}  \ \hat{V}({\rm d}y) \ {\rm d}u } && \\
& = & 2  \ \int_{\log r}^{\log {s}} e^{2 u}\int_0^\infty \psi(u (e^{u-y},{s},\theta )) \ e^{-d y}\ {\rm d}y\ {\rm d}u \\
& \overset{z=e^{u-y}}{=}& - 2 \ \int_{\log r}^{\log {s}} e^{2 u}\int_{e^{u}}^{0} \psi(u (z,{s},\theta )) z^d e^{-du} \ z^{-1} \ {\rm d}z\ {\rm d}u \\
& \overset{v=e^u}{=}& -2 \ \int_{ r}^{ {s}} v^{2}\int_{v}^0 \psi(u (z,{s},\theta )) \ z^{d-1} v^{-d}\ {\rm d}z\ v^{-1} \ {\rm d}v \\
& {=}&  2 \ \int_{ r}^{ {s}} v^{1-d}\int_{0}^v \psi(u (z,{s},\theta )) \ z^{d-1} \ {\rm d}z \ {\rm d}v.
\ex
Thus the characterising semi-group equation  (\ref{eq-semigroupbessel2}) resp. (\ref{eq-semigroupbrownianmotion}) becomes
\bex
u({r},{s},\theta) = \theta - 2 \ \int_{ r}^{ {s}} v^{1-d}\int_{0}^v \psi(u ({z},{s},\theta )) \ z^{d-1} \ {\rm d}z \ {\rm d}v.
\ex
Differentiation in $r$ gives
\bex
\frac{\partial}{\partial r} u({r},{s},\theta) &=& 2 r^{1-d} \int_0^r \psi(u({z},{s},\theta)) z^{d-1} \ {\rm d}z ,\\ 
\frac{\partial^2}{\partial r^2} u({r},{s},\theta) &=& 2 (1-d) r^{-d} \int_0^r \psi(u({z},{s},\theta)) z^{d-1} \ {\rm d}z + 2  \psi(u({r},{s},\theta)). 
\ex
Hence, we obtain the differential equation in (\ref{eq-pdebessel1}), i.e. for $\theta\geq 0$, 
\bex
\frac{1}{2}\frac{\partial^2}{\partial r^2} u({r},{s},\theta) + \frac{d-1 }{2 r} \frac{\partial}{\partial r} u({r},{s},\theta) &=&  \psi(u({r},{s},\theta)) \ \ \ 0<r\leq {s},\\
u(r,r,\theta)&=&\theta. 
\ex

\bibliographystyle{elsarticle-harv}

\bibliography{myrefsradialsuppnew}

%% Authors are advised to submit their bibtex database files. They are
%% requested to list a bibtex style file in the manuscript if they do
%% not want to use elsarticle-harv.bst.

%% References without bibTeX database:

\end{document}